\documentclass{amsart}
\usepackage{amsmath} \usepackage{amsfonts}
\usepackage{amstext}
\usepackage{amsbsy}
\usepackage{amsopn}
\usepackage{amsxtra}
\usepackage{upref}
\usepackage{amsthm}
\usepackage{amsmath}
\usepackage{amssymb}

\usepackage{mathrsfs}

\newtheorem{prop}{Proposition}[section]
\newtheorem{rem}{Remark}[section]
\newtheorem{lema}{Lemma}[section]
\newtheorem{defi}{Definition}[section]
\newtheorem{teo}{Theorem}[section]
\newtheorem{eje}{Example}[section]

\newtheorem*{claim*}{Claim}

\usepackage{euscript}

\DeclareMathSymbol{\varnothing}{\mathord}{AMSb}{"3F}
\renewcommand{\emptyset}{\varnothing}

\def\cK{\EuScript{K}}

\def\R{{\mathbb R}}

\def\N{{\mathbb N}}
\def\Z{{\mathbb Z}}
\def\M{{\mathcal M}}

\title{Phase transitions for suspension flows}
\date{\today}

\begin{thanks}
{We wish to thank Tom Kempton, Sergey Savchenko and  Mike Todd for many useful comments that improved the paper. G.I. was partially  supported by  Proyecto Fondecyt 1110040. T.J. wishes to thank Proyecto Mecesup-0711 for funding his visit to Chile}
\end{thanks}

\author{Godofredo Iommi} \address{Facultad de Matem\'aticas,
Pontificia Universidad Cat\'olica de Chile (PUC), Avenida Vicu\~na Mackenna 4860, Santiago, Chile}
\email{giommi@mat.puc.cl}
\urladdr{http://www.mat.puc.cl/\textasciitilde giommi/}
\author{Thomas Jordan} \address{The School of Mathematics, The University of Bristol, University Walk, Clifton, Bristol, BS8 1TW, UK}
\email{Thomas.Jordan@bristol.ac.uk}
\urladdr{http://www.maths.bris.ac.uk/~matmj}
\begin{document}

\begin{abstract}
This paper is devoted to study thermodynamic formalism for suspension flows defined over countable alphabets. We are mostly interested in  the regularity properties of the pressure function. We establish conditions for the pressure function to be real analytic or to exhibit a phase transition. We also construct an example of a potential for which the pressure has countably many phase transitions.
\end{abstract}

\maketitle
\section{Introduction}

In this paper we wish to study properties of suspension flows defined over countable Markov shifts. Suspension flows over expanding maps provide models for various examples of hyperbolic flows. Indeed, it was shown by Bowen \cite{bo1} and by Ratner \cite{ra} that
axiom A flows defined over compact manifolds can be coded as suspension flows over sub-shifts defined over finite alphabets. This type of coding allows for the proof of a great deal of fundamental results in ergodic theory. For instance, Bowen and Ruelle \cite{br} making use if these techniques studied the ergodic properties of axiom A flows and described in great detail the corresponding thermodynamic formalism.

On the other hand, suspension flows defined over countable Markov shifts can give models for several naturally occurring flows which are not uniformly hyperbolic or that are not defined over compact manifolds.  For example, Bufetov and Gurevich \cite{bg} have recently used such symbolic models for the Teichm\"{u}ller flows to show that the  Masur-Veech measure is the unique measure of maximal entropy. This result was extended by Hamenst\"adt \cite{h} who also used suspension flows over countable Markov shifts to model these  flows. It is also known, at least since the work of Artin \cite{ar}, that the geodesic flow over the modular surface can be coded with this type of symbolic models using the continued fraction map (see \cite{ku,se}). The positive geodesic flow is a restriction of the geodesic flow over the modular surface. In \cite{gk}  Gurevich and Katok   constructed a suspension flow defined over a countable Markov shift that coded this flow. Making use of this model they showed that the entropy of the flow is strictly smaller than one (recall that the entropy of the geodesic flow over the modular surface is one). The case of the geodesic flow on non-compact hyperbolic manifolds with cusps was studied by Babillot and Peign\'e \cite{bp}. In certain cases they obtain a representation that can be coded with a suspension flow over a countable alphabet. This allows for the proof of limit theorems.

The main focus of our study is the thermodynamic formalism associated to suspension flows defined over countable Markov shifts.  The pressure has been defined in this context, with different degrees of generality, in \cite{bi1,jkl, ke,sav}. The variational principle and other properties have been proved (see Section \ref{sec:pre}). However, the regularity properties of the pressure are poorly understood. If $\phi$ is a potential, the pressure function $t \mapsto P_{\Phi}(t \phi)$ is  a convex function, therefore it is differentiable except  in at most a countable set of points. The pressure function is said to have a \emph{phase transition} at the point $t=t_0$ if it is not analytic at that point.  It was shown by Bowen and Ruelle \cite{br} that in the context of suspension flows defined over sub-shifts of finite type on finite alphabets with a H\"older roof function, the pressure function for H\"older potentials is real analytic. That is, there are no phase transitions. The situation in our context is completely different. Indeed, in Theorem \ref{teo1} we establish sufficient conditions for a regular potential so that the corresponding pressure function exhibits a phase transition. We are able to determine exactly when the pressure may have phase transitions. However, we also give an example to show the behaviour can be much more complicated, in particular we construct  an example with countably many phase transitions.

It is worth pointing out that our results and examples hold, in particular, for suspension flows defined over the full-shift on a countable alphabet. In the discrete time case, that is countable Markov shifts, there is a combinatorial obstruction for the existence of phase transitions  for regular potentials. Indeed, for locally H\"older potentials (see sub-section \ref{cms} for precise definition) defined over the full-shift, the pressure function when finite is real analytic (see Theorem \ref{bip} or \cite{mu1, sa2}). The situation is completely different in the continuous time case and in Example \ref{count} we show it is possible for a regular potential to exhibit countably many phase transitions. The situation is similar to the one for multifractal analysis of Gibbs measures studied in \cite{HMU} and \cite{I}. Actually, with the techniques developed in this paper it should be possible to exhibit multifractal spectra with countably many phase transitions.

Our main result, Theorem  \ref{teo1}, holds for suspension flows defined over countable Markov shifts satisfying a combinatorial assumption, namely the BIP condition (see Section \ref{sec:pre} for a precise definition). However, adapting the well known procedure of inducing we are able to deal with systems not having this property. Indeed, in Theorem \ref{ind} we consider suspension flows defined over an arbitrary topologically mixing countable Markov shift and we establish conditions so that the pressure function is real analytic or has one phase transition. In this general setting  checking the condition of the theorem can be more complicated. That depends upon the combinatorics of the Markov shift on the base.

Even though our results are of a symbolic nature we stress that they can be readily applied to some geometric examples. In Section \ref{rem} we consider the positive geodesic flow. We establish conditions that ensure the existence of unique equilibrium measures. Moreover, we are able to prove that under certain assumptions the pressure function is real analytic or it can have one phase transition. A strong motivation for our work is to see how non-compactedness of the flow space can affect the properties of the pressure function and how this differs from the case of Axiom-A flows, where the pressure function is well understood (see \cite{br}).

We set out this paper as follows. In section 2.1 we recall the definition of pressure for Markov shifts defined over countable alphabets. We also state the results of Sarig \cite{sa1,sa2, sa3} and Mauldin-Urba\'{n}ski \cite{mu1, mu, mubook} regarding the behaviour of the pressure for countable Markov systems. Section 2.2 is devoted to define suspension flows and to discuss the relation between the spaces of invariant probability measures of the shift and the flow. The fact that this relation is not as good as in the finite state case is one of the sources of the more complicated behaviour exhibited by suspension flows on countable alphabets. In Section 3 we recall the definition of pressure for suspension flows given in \cite{bi1,jkl, ke,sav} and the properties that it satisfies. In Section 4 we state and prove one of our main results which establishes precise conditions for the existence of phase transitions. Section 5 is devoted to examples, of particular interest is Example \ref{count}  since it exhibits a countable number of phase transitions. Finally, in Section 6 we apply the above results to a geometric example, namely the positive geodesic flow and discuss the inducing procedure in this context.


\section{Preliminaries} \label{sec:pre}

\subsection{Thermodynamic formalism for countable Markov shifts} \label{cms}

Let $B$ be a transition matrix defined on the alphabet of natural numbers. That is, the entries of the matrix
$B=B(i,j)_{\N \cup \left\{0\right\}  \times \N \cup \{0\} }$  are zeros and ones (with no row and no column
made entirely of zeros). The countable Markov shift $(\Sigma_B, \sigma)$
is the set
\[ \Sigma_B := \left\{ (x_n)_{n \in \N \cup \{0\}} : B(x_n, x_{n+1})=1  \text{ for every } n \in \N  \cup \{0\} \right\}, \]
together with the shift map $\sigma: \Sigma  \to \Sigma $ defined by
$\sigma(x_0, x_1, \dots)=(x_1, x_2,\dots)$.

\begin{rem}
Analogously, we can define a two-sided countable Markov shift by
\[ \Sigma^*:= \left\{ (x_n)_{n \in \Z} : B(x_n, x_{n+1})=1  \text{ for every } n \in \Z \right\}, \]
together with the shift map  $\sigma: \Sigma^* \to \Sigma^*$ defined by $(\sigma x)_n=x_{n+1}$.
\end{rem}

Recall that the space $\Sigma_B$ is equipped with
the topology generated by the cylinder sets
\begin{equation*}
C_{a_0 \cdots a_n}= \{ x\in \Sigma_B: x_i=a_i \ \text{for
$i=0,\ldots,n$}\}.
\end{equation*}

In what follows we always assume $(\Sigma_B, \sigma)$ to be topologically mixing (see \cite[Section 2]{sa1} for a precise definition). When the context is clear we will simply write $(\Sigma, \sigma)$. Given a function $ \phi \colon \Sigma \to \R$ we define
\[ V_{n}(\phi):= \sup \{| \phi(x)- \phi(y)| : x,y\in \Sigma_B, \ x_{i}=y_{i}
\ \text{for $i=0,\ldots,n-1$} \},
\]
where $x=(x_0 x_1 \cdots)$ and $y=(y_0y_1 \cdots)$.  We say that $\phi$ has \emph{summable variation} if $\sum_{n=1}^{\infty} V_n(\phi)<\infty$.  We also say that
$\phi$ is \emph{locally H\"older} if there exist constants $K>0$ and
$\theta\in (0,1)$ such that $V_{n}( \phi) \le K \theta^{n}$ for all
$n\geq 1$. Each of the above properties implies that the function $\phi$ is uniformly continuous. A locally H\"older function has summable variations however neither of these conditions imply that the function is bounded.

The following notion of pressure for (non-compact) countable Markov shifts was introduced by Sarig  \cite{sa1}, generalising previous results by Gurevich \cite{gu1, gu2}.  Note that Mauldin and Urba\'nski \cite{mu1} gave a different definition of pressure for a narrower class of countable Markov shifts, however when both notions are defined they coincide.

\begin{defi} \label{presion}
Let $\phi \colon \Sigma \to \R$ be a function of summable variation. The
\emph{Gurevich pressure} of $\phi$  is defined by
\[
 P(\phi) = \lim_{n \to
\infty} \frac{1}{n} \log \sum_{x:\sigma^{n}x=x} \exp \left(
\sum_{i=0}^{n-1} \phi(\sigma^{i}x)\right)  \chi_{C_{i_{0}}}(x),
\]
where $\chi_{C_{i_{0}}}(x)$ is the characteristic function of the
cylinder $C_{i_{0}} \subset \Sigma$.
\end{defi}
It is possible to show that the limit always exists \cite{sa1}. Moreover, since
$(\Sigma,\sigma)$ is topologically mixing one can show that $P(\phi)$
does not depend on $i_0$.

\begin{rem} \label{formula}
Let $(\Sigma, \sigma)$ be the full-shift on countably many symbols, that is
\[ \Sigma:= \left\{ (x_n)_{n \in \N\cup \{0\}} : x_n \in \N \cup \{0\}\right\}.\]
If $\phi :\Sigma \to \R$ is a locally constant potential, that is
$\phi|C_n:= \log \lambda_n$, then there is a simple formula for the pressure (see, for example, \cite[Example 1]{bi1})
\begin{equation}
P(\phi )= \log \sum_{n=0}^{\infty} \lambda_n.
\end{equation}
\end{rem}

The Gurevich pressure satisfies the following approximation property (see \cite{sa1}),
\begin{teo}[Approximation property] Let $(\Sigma, \sigma)$ be a countable Markov shift and $\phi \colon \Sigma \to \R$ be a function of summable variations.   If
\[
\cK:= \{ K \subset \Sigma : K \ne \emptyset \text{ compact and }
\sigma\text{-invariant}\},
\]
then
\begin{equation}\label{*tan}
P ( \phi) = \sup \{ P_{K}( \phi) : K\in \cK \},
\end{equation}
where $P_{K}( \phi)$ is the classical topological pressure on
$K$;
\end{teo}
Moreover, this notion of pressure satisfies the variational principle (see \cite{mu1,sa1}),
\begin{teo}
Let $(\Sigma, \sigma)$ be a countable Markov shift and $\phi \colon \Sigma \to \R$ be a function of summable variations, then
\[P(\phi)= \sup \left\{ h(\nu) + \int \phi \text{d} \nu : \nu \in \M_{\sigma} \text{ and } - \int \phi \text{d} \nu < \infty \right\},\]
where $ \M_{\sigma} $ denotes the set of $\sigma-$invariant probability measures and $h(\nu)$ denotes the entropy of the measure $\nu$ (for a precise definition see  \cite[Chapter 4]{wa}).
\end{teo}
A measure $\nu \in \M_{\sigma}$ attaining the supremum, that is, $P(\phi)= h(\nu) + \int \phi \text{d} \nu$ is called \emph{equilibrium measure} for $\phi$. Buzzi and Sarig \cite{busa} proved that a potential of summable variations has at most one equilibrium measure.

We say that $\mu \in  \M_\sigma$ is a \emph{Gibbs measure}  for the
function $\phi \colon \Sigma \to \R$ if for some constants $P$,
$C>0$ and every $n\in \N$ and $x\in C_{a_0 \cdots a_n}$ we have
\[
\frac{1}C \le \frac{\mu(C_{a_0\cdots a_n})}{\exp (-nP + \sum_{i=0}^n
\phi(\sigma^k x))} \le C.
\]

 We say that a countable Markov shift $(\Sigma_B, \sigma)$, defined by the transition matrix $B(i,j)$ with $(i,j)\in \N \cup\{0\} \times \N \cup\{0\}  $, satisfies the \emph{BIP condition} if and only if  there exists $\{b_1 , \dots , b_n\}  \in \N \cup\{0\} $ such that for every $a \in \N \cup\{0\} $ there exists $i,j \in \N$ with $B(b_i, a)B(a,b_j)=1$. For this class of countable Markov shifts, introduced by Sarig \cite{sa2}, the thermodynamic formalism  is similar to that of sub-shifts of finite type defined on finite alphabets.
 The following theorem summarises results proven by Sarig in \cite{sa1,sa2} and by Mauldin and Urba\'{n}ski, \cite{mu1}.

\begin{teo} \label{bip}
Let $(\Sigma, \sigma)$ be a countable Markov shift satisfying the BIP condition and $\phi: \Sigma \to \R$ a locally H\"older potential. Then, there exists $t^* >0$ such that pressure function $t \to P(t\phi)$ has the following properties
\begin{equation*}
P(t \phi)=
\begin{cases}
\infty  & \text{ if  } t  < t^* \\
\text{real analytic } & \text{ if  } t > t^*.
\end{cases}
\end{equation*}
Moreover, if $t> t^*$, there exists a unique equilibrium measure for $t \phi$.
If $\sum_{n=1}^{\infty} V_{n}( \phi) < \infty$ and $P(\phi)< \infty$ then there exists a Gibbs measure for $\phi$.
\end{teo}

\subsection{Suspension flows and invariant measures}

Let $(\Sigma, \sigma)$ be a countable Markov shift and $\tau \colon \Sigma \to \R^+$ be a positive continuous function such that for every  $x\in\Sigma$ we have
$\sum_{i=0}^{\infty}\tau(\sigma^i x)=\infty$. Consider the space
\begin{equation}\label{shift}
Y= \{ (x,t)\in \Sigma  \times \R \colon 0 \le t \le\tau(x)\},
\end{equation}
with the points $(x,\tau(x))$ and $(\sigma(x),0)$ identified for
each $x\in \Sigma $. The \emph{suspension semi-flow} over $\sigma$
with \emph{roof function} $\tau$ is the semi-flow $\Phi = (
\varphi_t)_{t \ge 0}$ on $Y$ defined by
\[
 \varphi_t(x,s)= (x,
s+t) \ \text{whenever $s+t\in[0,\tau(x)]$.}
\]
In particular,
\[
 \varphi_{\tau(x)}(x,0)= (\sigma(x),0).
\]
In the case of two-sided Markov shifts we can define a suspension
flow $(\varphi_t)_{t\in \R}$ in a similar manner.

We denote by $\M_\Phi$ the space of $\Phi$-invariant probability
measures on $Y$. Recall that a measure $\mu$ on $Y$ is
\emph{$\Phi$-invariant} if $\mu(\varphi_t^{-1}A)= \mu(A)$ for every
$t \ge 0$ and every measurable set $A \subset Y$. We also consider
the space $\M_\sigma$ of $\sigma$-invariant probability measures on
$\Sigma $.  There is a strong relation between these two spaces of invariant measures.
Indeed, consider the space of $\sigma-$invariant measures for which $\tau$ is integrable,\begin{equation}
\M_\sigma(\tau):= \left\{ \mu \in \mathcal{M}_{\sigma}: \int \tau \text{d} \mu < \infty \right\}.
\end{equation}
Denote by $m$ the one dimensional Lebesgue measure and let $\mu \in \M_\sigma(\tau)$ then  it follows directly from classical results by Ambrose and Kakutani \cite{ak} that
\[(\mu \times m)|_{Y} /(\mu \times m)(Y) \in \M_{\Phi}.\]
The  behaviour of the map  $R \colon \M_\sigma \to \M_\Phi$, defined by
\begin{equation} \label{R}
R(\mu)=(\mu \times m)|_{Y} /(\mu \times m)(Y)
\end{equation}
 is closely related to the ergodic properties of the flow. Indeed, in the compact setting the map $R \colon \M_\sigma \to \M_\Phi$ is a bijection. This fact was used by Bowen and Ruelle \cite{br} to study and develop the thermodynamic formalism for Axiom A flows (these flows admit a compact symbolic representation). In
the general (non-compact) setting there are several difficulties that can arise. For instance,  the roof function $\tau$ need not to be bounded above. It is, therefore, possible for a measure $\nu \in \M_\sigma$ to be such that $\int \tau \text{d} \nu  = \infty$. In this situation the measure $\nu \times m$ is an infinite invariant measure for $\Phi$. Hence, the map $R(\cdot)$ is not well defined and this makes it harder to reduce the study of the thermodynamic formalism of the flow $\Phi$ to that of the shift $\sigma$. Another possible complication occurs if the roof function $\tau$ is not bounded away from zero. Then it is possible that for an infinite (sigma-finite) $\sigma-$invariant measure $\nu$  we have $\int  \tau\text{d}\nu <\infty$. In this case the measure $(\nu \times m)|_{Y} /(\nu \times m)(Y) \in \M_\Phi$. In such a situation,   the map $R$ is not surjective. Again, the fact that $R$ is not a bijective map makes it hard to translate problems from the flow to the shift.

The following can be obtained  directly from the results by Ambrose and Kakutani \cite{ak},
\begin{lema}
 If $\tau: \Sigma \to \R$ is bounded away from zero, the map $R \colon
\M_\sigma(\tau) \to \M_\Phi$  defined by
\begin{equation*}
R(\mu)=(\mu \times m)|_{Y} /(\mu \times m)(Y)
\end{equation*}
  is  bijective.
\end{lema}
Given a continuous function $g \colon Y\to\R$ we define the function
$\Delta_g\colon\Sigma\to\R$~by
\[
\Delta_g(x)=\int_{0}^{\tau(x)} g(x,t) \, \text{d}t.
\]
The function $\Delta_g$ is also continuous, moreover
\begin{equation} \label{rela}
\int_{Y} g \, \text{d}R(\nu)= \frac{\int_\Sigma \Delta_g\, \text{d}
\nu}{\int_\Sigma\tau \, \text{d} \nu}.
\end{equation}

\begin{rem}[Extension of potentials defined on the base] \label{exten}
Let $\phi \colon \Sigma  \to \R$ be a locally H\"older potential. It is shown in
\cite{brw} that there exists a continuous function $g \colon Y \to
\R$ such that $\Delta_g=\phi$. This provides a tool to construct
examples.
\end{rem}

\subsection{Abramov's formula} \label{abramov}
In this short subsection we recall a classical result by Abramov.  The entropy of a flow with respect to an invariant measure can be defined by  the entropy of the corresponding time one map. For the definition of entropy in the context of maps see \cite[Chapter 4]{wa}. In 1959 Abramov \cite{a} proved his well known result on the entropy of flows and Savchenko \cite[Theorem 1]{sav} proved an analogous result in the setting of this paper.
\begin{prop}[Abramov-Savchenko]
Let $\mu \in \M_{\Phi}$ be such that  $\mu=(\nu \times m)|_{Y} /(\nu \times m)(Y)$, where $\nu \in \M_{\sigma}$ then
\begin{equation}
h_{\Phi}(\mu)=\frac{h_{\sigma}(\nu)}{\int \tau \text{d} \nu}.
\end{equation}
\end{prop}

It follows from Savchenko's result that
\begin{lema}
Let $\mu \in \M_{\Phi}$ be such that $\mu=(\nu \times m)|_{Y} /(\nu \times m)(Y)$, we have that  $h_{\Phi}(\mu)= \infty$ if and only if
$h_{\sigma}(\nu)= \infty$.
\end{lema}

When the phase space is non-compact  there are several different notions of topological entropy of a flow, we will consider the following,
\begin{defi}
The topological entropy of the suspension flow $(Y ,\Phi)$ denoted by $h(\Phi)$ is defined by
\begin{equation*}
h(\Phi):= \sup \left\{  h_{\Phi}(\mu) : \mu \in \M_{\Phi}   \right\}.
\end{equation*}
\end{defi}

\section{Topological pressure for suspension semi-flows} \label{sec:Toppress}


A definition  of topological entropy for suspension semi-flows over \emph{countable} Markov shifts was first given by Savchenko \cite{sav}. He considered the case of roof functions depending only on the first coordinate, but not necessarily bounded away from zero. Topological entropy corresponds to topological pressure of the zero function. In \cite{bi1} Barreira and Iommi gave a definition of topological pressure for locally H\"older roof functions  bounded away from zero. Both in \cite{sav} and in \cite{bi1} the definitions are given implicitly.
 Recently, Kempton \cite{ke} and independently  Jaerisch, Kesseb\"ohmer and  Lamei \cite{jkl} gave a definition  of pressure for roof functions of summable variations that need not to be bounded away from zero. Moreover, they gave a closed formula for it.

The following theorem summarises  the above results,

\begin{teo} \label{pres}
Let $(\Sigma, \sigma)$ a topologically mixing countable Markov shift and $\tau:\Sigma \to \R$ a positive function bounded away from zero of summable variations.  Let $(Y, \Phi)$ be the suspension semi-flow over $(\Sigma, \sigma)$ with roof function $\tau$. Let $g:Y \to \R$ be a function such that $\Delta_g:\Sigma \to \R$ is of summable variations. Then the following equalities hold
\begin{eqnarray*}
P_{\Phi}(g)&:=&\lim_{t \to \infty} \frac{1}{t} \log \left(\sum_{\phi_s(x,0)=(x,0), 0<s \leq t} \exp\left( \int_0^s g(\phi_k(x,0)) \text{d}k \right) \chi_{C_{i_0}}(x) \right) \\
&=& \inf\{t \in \R : P_{\sigma} (\Delta_g - t \tau) \leq 0\} =\sup \{t \in \R : P_{\sigma} (\Delta_g - t \tau) \geq 0\} \\
&=& \sup \left\{ h_{\mu}(\Phi) +\int_Y g \text{d} \mu : \mu\in
\M_\Phi \text{ and } -\int_Y g \, \text{d}\mu <\infty \right\}\\
&=& \sup \{ P_{\sigma|K}( \phi) : K\in \cK \},
\end{eqnarray*}
where $\cK$ is the set of all compact and $\Phi-$invariant sets and  $P_K$ is the
classical topological pressure of the potential $\phi$ restricted to the
compact and $\sigma$-invariant set $K$.

\end{teo}

The third equality establishes the variational principle in this setting and the last the approximation property of the pressure.

\begin{rem}
If the roof function $\tau$ is only assumed to be positive (not necessarily bounded away from zero) the variational principle has only been proved for ergodic measures, that is
\[ P_{\Phi}(g)= \sup \left\{ h_{\mu}(\Phi) +\int_Y g \text{d} \mu : \mu\in
\M_\Phi \text{ ergodic and } -\int_Y g \, \text{d}\mu <\infty \right\}.\]
All the other results in Theorem \ref{pres} hold and have already been proved.
\end{rem}

Let now $g \colon Y \to \R$ be a continuous function such that
$\Delta_g$ is of summable variations. A measure $\mu\in \M_\Phi$ is
called an \emph{equilibrium measure} for $g$ if
\[
P_\Phi(g)= h_{\mu}(\Phi) +\int_Y g \text{d} \mu.
\]

The next theorem characterises potentials with equilibrium measures in the case that the roof function is bounded away from zero (see \cite[Theorem 4]{bi1}).

\begin{teo}\label{faa}
Let $\Phi$ be a suspension semi-flow on $Y$ over a countable Markov
shift  $(\Sigma, \sigma)$  and roof function $\tau$ of summable variations and bounded away from zero. Let $g \colon Y \to \R$ be a continuous function such
that $\Delta_g$ is of summable variations. Then
there is an equilibrium measure $\mu_g\in \M_\Phi$ for $g$ if  and only if
we have that $P(\Delta_g -P_{\Phi}(g) \tau)=0$ and there exists
 an equilibrium measure $\nu_g\in \M_\sigma(\tau)$ for $\Delta_g -P_{\Phi}(g) \tau$.
\end{teo}

\subsection{From flows to semi-flows}
A classical result by Sinai allow us to reduce the study of the thermodynamic formalism for flows to the study of the thermodynamic formalism for semi-flows.

Let  $(\Sigma^*, \sigma)$ be a two-sided Markov shift. Two continuous functions $\phi,
\gamma \in C(\Sigma^*)$ are said to be \emph{cohomologous} if
there exists a continuous function $\psi \in C(\Sigma^*)$ such
that $\phi= \gamma +\psi \circ \sigma -\psi$. The following
statement is due to Sinai (see \cite[Proposition 1.2]{pa}) for
H\"older continuous functions and to Coelho and Quas \cite{cq} for
functions of summable variation.

\begin{prop}
If $\phi \in C(\Sigma^*)$ has summable variation, then there
exists $\gamma \in C(\Sigma^*)$ cohomologous to $\phi$ such that
$\gamma(x)=\gamma(y)$ whenever $x_i=y_i$ for all $i \geq 0$ (that is,
$\gamma$ depends only on the future coordinates).
\end{prop}

Furthermore, if the function $\phi$ has summable variation, then the
same happens with~$\gamma$. Denote by  $(\Sigma, \sigma)$ be the one-sided Markov shift associated to  $(\Sigma^*, \sigma)$. We note that $\gamma$ can be canonically
identified with a function $\varphi\colon\Sigma \to\R$, and
$P_{\Sigma^*}(\phi)= P_{\Sigma}(\varphi)$.

Note that the pressure function is invariant under cohomology, so is the existence of equilibrium measures and in general the thermodynamic formalism.

\section{Phase transitions}

It is a direct consequence of the approximation property of the pressure (see Theorem \ref{pres}) that
$P_\Phi(\cdot)$ is convex since it is the supremum of convex functions. In particular
the function $t \to P_{\Phi}(tg)$ is differentiable except in at most a countable set of points.
We say that the pressure function $t \to  P_{\Phi}(tg)$ has a \emph{phase transition} at the point $t=t_0$ if it is not analytic at $t=t_0$. We say that the pressure has a \emph{first order phase transition}  if the function is not differentiable at $t=t_0$.  In the context of countable Markov shifts, Sarig \cite[Corollary 4]{sa3}, showed that if $(\Sigma, \sigma)$ satisfies the BIP condition and $\phi:\Sigma \to \R$ is a locally H\"older potential then, when finite, the pressure function $t \to P(t \phi)$ is  real analytic.
  In this Section we prove that in the case of flows the situation is different. Indeed, the combinatorial structure of the base (e.g. satisfying the BIP condition) does not rule out the existence of phase transitions.

Let $(\Sigma, \sigma)$ be a countable Markov shift satisfying the BIP condition and let $\tau:\Sigma \to \R$ a positive  locally H\"older function for which there exists a real number $s_{\infty} \in \R$ satisfying
\begin{equation*}
P(-s\tau)=
\begin{cases}
\infty & \text{ if } s <s_{\infty};\\
\text{finite} & \text{ if } s > s_{\infty}.
\end{cases}
\end{equation*}
In particular, the above assumptions implies that $\tau$ is not bounded above.  Denote by $(Y, \Phi)$  the suspension semi-flow defined over the $(\Sigma, \sigma)$ with  roof function  $\tau$. In view of the above assumptions this semi-flow has finite entropy. Let $g:Y \to \R$ be a potential, we will study regularity properties and phase transitions of the pressure function $t \to P_{\Phi}(tg)$.

The main case we will consider is when
\begin{equation*}
\lim_{n \to \infty} \frac{\sup \{ \Delta_g(x): x \in C_n\}}{\inf \{ \tau(x): x \in C_n\}}=0.
\end{equation*}
Note that the case when for $a\in\R$ we have that
\begin{equation*}
\lim_{n \to \infty} \frac{\sup \{ \Delta_g(x): x \in C_n\}}{\inf \{ \tau(x): x \in C_n\}}=a
\end{equation*}

can easily be transfered to this case by subtracting the constant $a$ from $g$. In this setting, for positive locally H\"{o}lder potentials we can specify exactly when we get phase transitions.

\begin{teo} \label{teo1}
Let $(Y, \Phi)$ be  the suspension semi-flow of finite entropy defined over a shift on a countable alphabet satisfying the BIP condition where the roof function $\tau$ is bounded away from $0$. Let $g:Y \to  \R$ be a potential such that
$\Delta_g:\Sigma \to \R$ is locally H\"older and
\begin{equation*}
\inf \left\{ \frac{\int \Delta_g \text{d} \nu}{\int \tau \text{d} \nu} : \nu \in \M_{\sigma} \right\} = \lim_{n \to \infty} \frac{\sup \{ \Delta_g(x): x \in C_n\}}{\inf \{ \tau(x): x \in C_n\}}=0.
\end{equation*}
We then have the following situation
\begin{enumerate}
\item\label{part1}
For all $t \in \R$ we have that  $P_{\Phi}(tg)\geq s_{\infty}$ and in the case that $\Delta_g(x)>0$ for all $x\in\Sigma$ we have that $\lim_{t\to-\infty}P_{\Phi}(tg)=s_{\infty}$.
\item\label{part2}
If for all $t\in\R$ we have that $P(t\Delta_g-s_{\infty}\tau)=\infty$ then there is no phase transition and the function $t\to P(tg)$ is real analytic.
\item\label{part3}
If $\Delta_g(x)>0$ for all $x\in\Sigma$ and there exists $t^* \in \R$ such that $P(t^*\Delta_g-s_{\infty}\tau)<\infty$ then there exists $t_0 \in \R$  such that $t_0=\sup\{t:P(t \Delta_g-s_{\infty}\tau)\leq 0\}$, and
\begin{eqnarray*}
P_{\Phi}(tg)=
\begin{cases}
\text{real analytic and strictly convex } & \text{ if } t>t_0;\\
s_{\infty}& \text{ if } t \leq t_0.
\end{cases}
\end{eqnarray*}
\end{enumerate}
\end{teo}

\begin{proof}

We divide the proof of this result in several Lemmas. The assumptions are always the same as in Theorem \ref{teo1}

\begin{lema} \label{1}
For every $t \in \R$ we have that $P_{\Phi}(t g) < \infty$.
\end{lema}

\begin{proof}
 Recall that the pressure satisfies the variational principle and that
the semi-flow has finite entropy. Therefore, in order to bound the pressure we just need to bound the integral of $g$. However since
$$\lim_{n \to \infty} \frac{\sup \{ \Delta_g(x): x \in C_n\}}{\inf \{ \tau(x): x \in C_n\}}=0$$
it is easy to see
that there exists $K \in \R$ such that for every $x \in \Sigma$ we have
\begin{equation*}
\left|\frac{\Delta_g(x)}{\tau(x)}\right| \leq K.
\end{equation*}
In particular, for every $\nu \in \M_{\sigma}$ we have
\begin{equation*}
\left|\frac{\int \Delta_g \text{d} \nu}{\int \tau \text{d} \nu}\right| \leq K.
\end{equation*}
Therefore, for every $\mu \in \M_{\Phi}$ we have $\left|\int g \text{d} \mu\right| \leq K$. The result now follows.
%
\end{proof}

\begin{lema} \label{l_bound}
For every $t \in \R$ we have that
\begin{equation*}
P(t\Delta_g+s\tau)=
\begin{cases}
\infty & \text{ if } s<s_{\infty};\\
\text{finite } & \text{ if } s>s_{\infty}.
\end{cases}
\end{equation*}
In particular
$P_{\Phi}(t g) \geq  s_{\infty}$.
\end{lema}
\begin{proof}
By the definition of $s_{\infty}$ it follows that
\begin{equation*}
P(-s\tau)=
\begin{cases}
\infty & \text{ if } s<s_{\infty};\\
\text{finite } & \text{ if } s>s_{\infty}.
\end{cases}
\end{equation*}
Thus, by the variational principle we have that for any sequence of invariant measures $\{\mu_n\}_{n\in\N}$, if $\lim_{n\to\infty}\int\tau\text{d}\mu_n=\infty$ then $\limsup_{n\to\infty}\frac{h(\mu_n)}{\int\tau\text{d}\mu_n}\leq s_{\infty}$. Moreover, there exists a sequence of invariant measures $\{\nu_n\}_{n\in\N}$ such that $\lim_{n\to\infty}\int\tau\text{d}\nu_n=\infty$ and
$\lim_{n\to\infty}\frac{h(\nu_n)}{\int\tau\text{d}\nu_n}= s_{\infty}$ (see \cite[Lemma 2.5]{jr} and note that while it is only stated for the full shift there it could be easily extended to the BIP case).

We now fix $t\in\R$ and let $s<s_{\infty}$. We then have that
\begin{eqnarray*}
P(t\Delta_g+s\tau)\geq \lim_{n\to\infty} \left( h(\nu_n)+t\int\Delta_g\text{d}\nu_n-s\int\tau\text{d}\nu_n \right) =\\
 \lim_{n\to\infty} \left( \int\tau\text{d}\nu_n \left( \frac{h(\nu_n)}{ \int\tau\text{d}\nu_n} +t \frac{\int\Delta_g\text{d}\nu_n}{ \int\tau\text{d}\nu_n} -s \right)  \right).
\end{eqnarray*}
 As noted in Lemma \ref{1} the quotient $\frac{\int\Delta_g\text{d}\nu_n}{ \int\tau\text{d}\nu_n}$ is bounded above. Therefore
\begin{eqnarray*}
P(t\Delta_g+s\tau)\geq \lim_{n\to\infty} \left( \int\tau\text{d}\nu_n \left( \frac{h(\nu_n)}{ \int\tau\text{d}\nu_n} +t \frac{\int\Delta_g\text{d}\nu_n}{ \int\tau\text{d}\nu_n} -s \right)  \right) = \infty.
\end{eqnarray*}

Let us now fix $s>s_{\infty}$.  If there exists a sequence of invariant measures $\{ \mu_n\}_{n\in\N}$ such that
\begin{equation*}
\lim_{n\to\infty} \left( h(\mu_n)+t\int\Delta_g\text{d}\mu_n-s\int\tau\text{d}\mu_n \right)=\infty
\end{equation*}
 then there exists a sequence of measures $\{ \mu_n \}_{n \in \N}$ such that $\int\tau\text{d}\mu_n=\infty$ and $\lim_{n\to\infty}\frac{h(\mu_n)}{\int\tau\text{d}\mu_n}>s_{\infty}$ which is a contradiction. Thus, the first part of the lemma follows. The second part is a direct consequence of the definition of $P_{\Phi}$.
\end{proof}

The following lemma completes the proof of part\ref{part1} of
Theorem \ref{teo1}.

\begin{lema}
We have that if $\Delta_g(x)>0$ for all $x\in\Sigma$ then
\begin{equation*}
\lim_{t \to - \infty} P_{\Phi}(t g)=s_{\infty}.
\end{equation*}
\end{lema}
\begin{proof}
Note that
\[ \inf \left\{ \frac{\int \Delta_g \text{d} \nu}{\int \tau \text{d} \nu} : \nu \in \M_{\sigma} \right\} =0\]
implies that the slope of $P_{\Phi}(tg)$ as $t \to -\infty$ tends to zero.  Therefore, in order to compute the pressure for values of $t$ converging to $-\infty$ we have to consider measures $\nu_n \in \mathcal{M}_{\sigma}$ with $\lim_{n \to \infty}Ê\int \tau \text{d} \nu_n = \infty$.

Moreover, if $\{ \nu_n\}_{n\in \N} \in \M_{\sigma}$ is a sequence of invariant measures such that
\[\lim_{n \to \infty}Ê\int \tau \text{d} \nu_n = \infty,\]
then
\[ \lim_{n \to \infty}Ê\frac{h(\nu_n)}{\int \tau \text{d} \nu_n}  \leq s_{\infty} \quad \text{ and } \quad
 t \lim_{n \to \infty}Ê\frac{\int \Delta_g \text{d} \nu_n}{\int \tau \text{d} \nu_n} =0.\]
 Therefore,
 \begin{equation*}
\lim_{t \to - \infty} P_{\Phi}(t g)=  s_{\infty}.
\end{equation*}
\end{proof}

We now consider when the function $t\to P_{\Phi}(tg)$ is real analytic.

\begin{lema}\label{analytic}
If we let $A\subset\R$ be an interval such that for all $t\in\R$ we have that $P(t\Delta_g-s_{\infty}\tau)>0$ then the function $t\to P_{\Phi}(tg)$ is real analytic on $A$.
\end{lema}
\begin{proof}
Recall that the function $f:(s_{\infty},\infty)\times \R\to \R$ defined by $f(t,s)=P(t\Delta_g-s\tau)$, is finite and so is real analytic in both variables $s$ and $t$ (see for example \cite{su}). If $t\in A$ then the function
$s \to f(s,t)$ is decreasing and $\lim_{s \to \infty} f(s,t)= -\infty$. Therefore, there exists an unique $s_0>s_{\infty}$ such that $f(t,s_0)=0$ and  $\frac{\partial f}{\partial s}(s_0)<0$. This implies that $P_{\Phi}(tg)=s_0$. Thus, by the Implicit Function Theorem, the function $t\to P_{\Phi}(tg)$ is real analytic on $A$.
\end{proof}

Note that Part \ref{part2} of Theorem \ref{teo1} immediately follows since if $P(t\Delta_g-s_{\infty}\tau)=\infty$ for all $t\in\R$ then in the above Lemma $A=\R$. To prove \ref{part3} of Theorem \ref{teo1} we need the following lemma.

\begin{lema}\label{keyvalue}
If $\Delta_g(x)>0$ for all $x\in\Sigma$ and there exists $t^* \in \R$ such that $P(t^*\Delta_g-s_{\infty}\tau)<\infty$ then there exists a finite value  $t_0\in\R$ where
$$t_0=\sup\{t:P(t\Delta_g-s_{\infty}\tau)\leq 0\}.$$
\end{lema}
\begin{proof}
By the properties of the pressure function it suffices to show that if there exists $t^* \in \R$ where $P(t^*\Delta_g-s_{\infty}\tau)<\infty$ then $\lim_{t\to-\infty}P(t^*\Delta_g-s_{\infty}\tau)=-\infty$. Furthermore since $\Delta_g> 0$ we can assume that $t^*<0$. We consider an upper bound for the pressure function $Q:\R^2\to\R\cup\{\infty\}$ given by
$$Q(s,t)=\log\left(\sum_{i=0}^{\infty}\exp \left(\sup \left \{t\Delta_g(x)-s\tau(x) : x\in C_i \right\} \right) \right).$$
Note that $Q(s_\infty,t)=\infty$ if and only if $P(t\Delta_g-s_{\infty}\tau)=\infty$ (see \cite{mubook} Proposition 2.1.9  ) and $Q(s,t)\geq P(t\Delta_g-s\tau)$.

Thus if $P(t^*\Delta_g-s_{\infty}\tau)<\infty$ then $Q(t^*\Delta_g-s_{\infty}\tau)<\infty$. Hence, for any $\epsilon>0$ there exists $N \in \N$ such that
$$
\sum_{i=N}^{\infty}\exp(\sup\{t^* \Delta_g(x)-s_{\infty}\tau(x):x\in C_i\})\leq\epsilon.
$$
Furthermore since $\Delta g>0$ there will exist $t<t^*$ such that
\begin{equation} \label{serieQ}
\sum_{i=N}^{\infty}\exp(\sup\{t \Delta_g(x)-s_{\infty}\tau(x):x\in C_i\})\leq\epsilon.
\end{equation}
and
\begin{equation} \label{parte-finita}
\sum_{i=0}^{N}\exp(\sup\{t\Delta_g(x)-s_{\infty}\tau(x):x\in C_i\})\leq\epsilon.
\end{equation}
 Therefore, we have that
$$\left(\sum_{i=0}^{\infty}\exp \left(\sup\{t\Delta_g(x)-s_{\infty}\tau(x):x\in C_i\} \right)\right)\leq 2\epsilon$$
and the result follows.
\end{proof}
If we take $t_0$ as defined in Lemma \ref{keyvalue} and $A$ as defined in Lemma \ref{analytic} then $A=(t_0,\infty)$ and it follows immediately that $t\to P_{\Phi}(tg)$ is analytic on $A$. For $t\in (-\infty, t_0)$ we have by Lemma \ref{l_bound} that if $s<s_{\infty}$ then $P(t\Delta_g-s\tau)=\infty$  and $P(t\Delta_g-s_{\infty}\tau)<0$ which means that $P_{\Phi}(tg)=s_{\infty}$. This completes the proof of Theorem \ref{teo1}.

\end{proof}

\begin{rem}
We would like to point out a mistake in \cite[Proposition 8]{bi1}. It is claimed in there that if $g:Y \to \R$ is a  bounded  potential and $(\Phi, Y)$ is a suspension semi-flow defined over a BIP shift then there are no phase transitions. The error in the proof of  \cite[Proposition 8]{bi1} is in \cite[Proposition 7]{bi1}, where it is claimed that, under these assumptions, the equation $P(\Delta_g -t\tau)=0$ always has a root. This is not true as  Theorem \ref{teo1} and the examples below show.
\end{rem}

\section{Examples}
We now give some examples both to illustrate Theorem \ref{teo1} and to look at what can happen when the assumptions of Theorem \ref{teo1} are not met. Our first two examples are examples of  regular potentials which satisfy the conditions of Theorem \ref{teo1} and for one of which the pressure function exhibits no phase transition and for the other potential  one phase transition. It should be noticed that this is a phase transition of positive entropy. This type of phase transitions have received some attention recently, see for example  \cite{dgr,it}. The example we exhibit here is the first one of this type in the context of flows. Our final example shows that without the assumptions used in Theorem \ref{teo1} there is a possibility of infinitely many phase transitions.

\begin{eje}[No phase transitions]\label{nophase}
Let $(\Sigma, \sigma)$ be the full-shift on a countable alphabet, say $\N\cup\{0\}$ and $\tau:\Sigma \to \R$ be the roof function defined by
\begin{equation*}
\tau(x):= \log (n + 2) \quad \text{ if } \quad x \in C_n.
\end{equation*}
Denote by $(Y, \Phi)$ the suspension semi-flow with base $(\Sigma, \sigma)$ and roof function $\tau$. The entropy of this flow is finite and
\[h(\Phi)= \inf \{ s \in \R : P(-s\tau) \leq 0 \} > 1, \]
Consider a function $g:Y \to \R$ such that
\begin{equation*}
\Delta_g(x)=\log\log\log (n+1) \text{ if } \quad x \in C_n.
\end{equation*}
We then have that all the assumptions of Theorem \ref{teo1} are satisfied, $s_{\infty}=1$ and for any $t\in\R$
$$P(t\Delta g-\tau)=\log\left(\sum_{n=1}^{\infty} \frac{(\log\log n)^t}{n+1}\right)=\infty.$$
Thus, it follows in this case that the function $t\to P_{\Phi}(tg)$ is real analytic and strictly convex for all $t\in\R$. Moreover,  $\lim_{t\to-\infty}P_{\Phi}(tg)=1$.
\end{eje}

\begin{eje}[One Phase Transition] \label{1phase}
Let $(\Sigma, \sigma)$ be the full-shift on a countable alphabet, say $\N\cup\{0\}$. Let $k \in \N$ be such that
\[ \sum_{n=k}^{\infty} \frac{1}{n (\log n)^2} < 1. \]
Let $\tau:\Sigma \to \R$ be the roof function defined by
\begin{equation*}
\tau(x):= \log (n + k) \quad \text{ if } \quad x \in C_n.
\end{equation*}
Denote by $(Y, \Phi)$ the suspension semi-flow with base $(\Sigma, \sigma)$ and roof function $\tau$. The entropy of this flow is
\[h(\Phi)= \inf \{ s \in \R : P(-s\tau) \leq 0 \}, \]
that is $h(\Phi) \in (1, 2)$. Indeed, there exists $s \in (1, 2)$ such that $P(-s \tau)=0$.
Consider a function $g:Y \to \R$ such that
\begin{equation*}
\Delta_g(x)= \log (\log (n+k))  \quad \text{ if } \quad x \in C_n.
\end{equation*}
Again we have that $s_{\infty}=1$ and all the assumptions of Theorem \ref{teo1} are satisfied. For $t\in\R$ we have that
$$P(t\Delta_g-\tau)=\log \sum_{n=k}^{\infty}\frac{(\log n)^t}{n}$$
and so for all $t\geq -1$ we have that $P(t\Delta_g-\tau)=\infty$. However for all $t<-1$ we have that $P(t\Delta_g-\tau)<\infty$ and in particular $P(-2\Delta_g-\tau)<0$. Thus by Theorem \ref{teo1} the pressure function $t \to P_{\Phi}(tg)$ exhibits a phase transition at a point $t_0 \in (-2, -1)$. Moreover,
 \begin{equation*}
 P_{\Phi}(t g)=
 \begin{cases}
 1 & \text{ if } t < t_0 ; \\
 \text{real analytic and strictly convex}  & \text{ if } t \geq t_0.
   \end{cases}
 \end{equation*}
 In particular, the function $g$ exhibits a phase transition of positive entropy at the point $t=t_0$.

  Note that for $t < t_0$ we have
 \begin{equation*}
 P(t\Delta_g-P_{\Phi}(t g) \tau)=P(t\Delta_g-\tau)=\log \sum_{n=k}^{\infty}\frac{(\log n)^t}{n} <0.
\end{equation*}
It is a direct consequence of Theorem \ref{faa} that for every  $t < t_0$ the potential $tg$ does not have an equilibrium measure. On the other hand, if $t >t_0$ we have that
 \begin{equation*}
 P(t\Delta_g-P_{\Phi}(t g) \tau)=0.
\end{equation*}
Moreover, there exists a Gibbs measure $\nu_t$ corresponding to $t\Delta_g-P_{\Phi}(t g) \tau$ and
 \begin{equation*}
\int \tau d \nu_t = \sum_{n=1} \log(n+k) \nu_t(C_n)= \sum_{n=k} \frac{(\log n)^{t+1}}{n^{ P_{\Phi}(t g)}}.
\end{equation*}
Since $P_{\Phi}(t g)>1$ we have that $\tau \in L^1(\nu_t)$. Therefore, as consequence of Theorem \ref{faa}, we have that for every  $t > t_0$ the potential $tg$ does  have a unique equilibrium measure. If $t=t_0$ the potential $t_0g$ does not have an equilibrium measure. Indeed, even if $k$ is chosen so that $P(t\Delta_g-P_{\Phi}(t g) \tau)=P(t\Delta_g-\tau)=0$, we have that
since $t_0 \in (-2,-1)$ then   $\tau \notin L^1(\nu_{t_0})$. However, it is possible to slightly modify this example so that $t_o g$ has an equilibrium state.
\end{eje}

\begin{eje}[Countably many phase transitions]\label{count}
In Theorem \ref{teo1} we assumed that the roof function dominated the potential $\Delta_g$. This leads to the simple asymptotic behaviour
\begin{equation*}
\lim_{n \to \infty} \frac{\sup \{ \Delta_g(x): x \in C_n\}}{\inf \{ \tau(x): x \in C_n\}}=0.
\end{equation*}
In this example we show that without this assumption  much more complicated behaviour is possible. In particular, by dropping this assumption, we can construct a potential $g$ defined over a  suspension semi-flow for which the pressure function $t \to P_{\Phi}(t g)$ has countably many phase transitions. Again we work in the case where both $\Delta g$ and $\tau$ are locally constant functions.

We start by partitioning $\N\cup\{0\}$ as follows. We let
$$A_1=\{n\in\N:n\neq k^2 \text{ for any }k\in\N\}\cup\{0\}$$
and inductively define $A_{k+1}=\{n\in\N:\sqrt{n}\in A_k\}$.
We then have that $\N=\cup_k A_k$ and $A_i\cap A_j=\emptyset$ if $i\neq j$.

Let $(\Sigma, \sigma)$ be the full-shift on a countable alphabet. We consider the roof function $\tau:\Sigma\to \R$ to be defined by
$$\tau(x):=\log (n+k_i) \text{ if } x\in C_n \text{ and } n\in A_i,$$
where $k_i >0 $ are some fixed constant. Denote by $(Y, \Phi)$ the corresponding suspension semi-flow.

Let us define our function $g:Y \to \R$ such that
$$\Delta_g(x)=  c_i\tau(x)+\log\log (n+k_i) \text { if } x\in C_n \text{ and } n\in A_i$$
where  $c_0=0$, $c_1=-\frac{1}{4}$ and for $i>1$ we define $c_{i+1}=c_i-\frac{1}{2^{i}(i+1)}$.

We then have that if  $\sum_{n\in A_i}\frac{\left(\log(n+k_i)\right)^t}{(n+k_i)^{s-c_it}}<\infty$ for all $i \in \N$ then
$$P(t\Delta_g-s\tau)=\log \left( \sum_{i=1}^{\infty}\sum_{n \in A_i}\frac{\left(\log(n+k_i)\right)^t}{(n+k_i)^{s-c_it}} \right)$$
and otherwise $P(t\Delta_g-s\tau)=\infty$.

\begin{lema}\label{Plb}
With $\tau$ and $g$ defined as above we have that for all $t<0$
$$P_{\Phi}(tg)\geq \max_{i \in \N} \left\{\frac{1}{2^i}+c_it \right\}.$$
\end{lema}

\begin{proof}
Let $t\leq 0$ and $\delta>0$ and consider $s= \frac{1}{2^{i-1}}+c_it- \delta$. We have that
\begin{equation*}
\sum_{n \in A_i} \frac{\left(\log(n+k_i)\right)^t}{(n+k_i)^{s-c_it}}=
\sum_{n \in A_i} \frac{\left(\log(n+k_i)\right)^t}{(n+k_i)^{\frac{1}{2^{i-1}}- \delta}}= \infty.\end{equation*}
Therefore $P(t\Delta_g -s\tau)= \infty$. Since the above holds for any positive value of $\delta$, for every $i \in \N$ we have that
$$\frac{1}{2^i}+c_it \leq P_{\Phi}(tg).$$
\end{proof}
%

\begin{lema}\label{ki}
We can choose values $k_i \in \R$ such that for $t\leq -\frac{3}{2}$
$$P_{\Phi}(tg)=\max_{i \in \N} \left\{\frac{1}{2^i} +c_it \right\}.$$
\end{lema}

\begin{proof}
For each $i \in \N$ we fix $0<\delta_i<\frac{1}{2}$ and choose $k_i$ such that
$$\sum_{n \in A_i}\frac{1}{(n+k_i)^{2^{-(i-1)}}}\log(n+k_i)^{-(1+\delta_i)}<\frac{1}{2^i}.$$
Thus for any $t<-\frac{3}{2}$ and $\alpha \geq \frac{1}{2^{i-1}}$ we have that
$$\sum_{n \in A_i}\frac{1}{(n +k_i)^{\alpha}}(\log(n+k_i))^t<\frac{1}{2^i}.$$

Choose $n\in \N$ such that $\frac{1}{2^n}-c_nt=\max_i\{\frac{1}{2^i}-c_it\}$ and $s>\frac{1}{2^n}-c_nt$. In particular,  for every $i \in \N$ we have that
$s-c_i t> 2^{-(i-1)}$. Therefore, if we fix $t<-3/2$ we have that

\[ \sum_{n \in A_i}\frac{\left(\log(n+k_i)\right)^t}{(n+k_i)^{s-c_it}} \leq
 \sum_{n \in A_i}\frac{\left(\log(n+k_i)\right)^t}{(n+k_i)^{1/2^{i-1}}} \leq \frac{1}{2^i}.\]

Thus,

\begin{eqnarray*}
P(t\Delta_g-s\tau)= \log \left( \sum_{i=1}^{\infty}\sum_{n \in A_i}\frac{\left(\log(n+k_i)\right)^t}{(n+k_i)^{s-c_it}} \right) \leq \log \left( \sum_{i=1}^{\infty} \frac{1}{2^i}  \right) < 0.
\end{eqnarray*}
%

Therefore,  $P_{\Phi}(tg)<s$. Thus combining this with Lemma \ref{Plb} we have that for $t\leq -3/2$
$$P_{\Phi}(tg)=\max_{i \in \N} \left\{\frac{1}{2^i}-c_it \right\}.$$
\end{proof}

Thus if we choose a sequence $(k_i)_i$ as in Lemma \ref{ki} we have that for $t\leq-\frac{3}{2}$
$$P_{\Phi}(tg)=\max_i \left\{\frac{1}{2^i}-c_it \right\}$$
where $c_0=0$, $c_1=-\frac{1}{4}$ and for $i>1$ we define $c_{i+1}=c_i-\frac{1}{2^{i}(i+1)}$. We can then see that for $i\geq 1$ for $t\in [i+1,i+2]$ we have that
$$P_{\Phi}(tg)=\frac{1}{2^i}-c_it.$$
Thus, for $t\leq -\frac{3}{2}$ we have that $t\to P_{\Phi}(tg)$ is a piecewise linear function which is non-differentiable at integer values., That is, the pressure function exhibits a countable number of first order phase transitions.
\end{eje}

\section{Remarks} \label{rem}
This last section is divided in two. In the first part we illustrate how the the results obtained in the symbolic system can be applied to flows defined over manifolds. The particular example we consider is the positive geodesic flow. In the second part we discuss a strategy to study suspension flows for which the combinatorial assumption we have made so far on the shift (that of being BIP) is not satisfied. The procedure we propose is that of inducing.

\subsection{The positive geodesic flow on the modular surface}

%
 Denote by $\mathcal{H}=\{z \in \mathbb{C}: \text{Im }  z >0  \}$ the upper half-plane endowed with the hyperbolic metric. Geodesics in $\mathcal{H}$ are either semi-circles which meet the boundary perpendicularly or vertical straight lines.
The geodesic flow of $\mathcal{H}$, denote by $\{\overline{\psi}_t\}$, is the flow on the unit tangent bundle,
 $T^1\mathcal{H}$, of $\mathcal{H}$ which moves $\omega \in T^1\mathcal{H}$ along the geodesic it determines at unit speed.

The group of M\"obius transformations acting on $\mathcal{H}$ by orientation preserving isometries can be identified with the group $ \text{PSL}(2, \mathbb{Z})$. The \emph{modular surface} is defined by $M=\text{PSL}(2, \mathbb{Z})  \backslash  \mathcal{H}$, which is a (non-compact) surface of  constant negative curvature. Topologically, is a sphere with one cusp and two singularities. Note that $T^1\mathcal{H}$ can be identified with $\text{PSL}(2, \mathbb{R})$ by sending $\omega=(z, \zeta) \in T^1\mathcal{H} $ onto the unique $ g \in  \text{PSL}(2, \mathbb{R})$ such that $z=g(i)$ and $\zeta=g'(z)(\iota)$, where $\iota$ is the unit vector at the point $i$ to the imaginary axis pointing up. In this coordinate system the geodesic flow takes the form
\begin{equation*}
\overline{\psi}_t
\left(  \begin{array}{cc}   a & b \\ c & d \end{array} \right) = \left(  \begin{array}{cc}   a & b \\ c & d \end{array} \right) \left(  \begin{array}{cc}   e^{t/2} & 0 \\ 0 & e^{-t/2} \end{array} \right).\end{equation*}
The geodesic flow  $\{\overline{\psi}_t\}$ on $\mathcal{H}$ descends to the geodesic flow
 $\{\psi_t\}$ on $M$ via the projection $\pi:T^1\mathcal{H} \mapsto T^{1}M$. We will be interested in an invariant sub-system of  $\{\psi_t\}$. In order to define it we need to
 consider the fundamental region for $\text{PSL}(2, \mathbb{Z})$ given by
\begin{equation*}
F=\{z \in \mathcal{H} : |z| \geq 1, | \text{Re } z | \leq 1/2 \},
\end{equation*}
whose sides are identified by the generators of  $\text{PSL}(2, \mathbb{Z})$, $T(z)=z+1$ and $S(z)=-1/z$ (see \cite[p.55]{ka}). Any geodesic can be represented by a series of segments in $F$. We will only be interested in oriented geodesic that do not go to the cusps of $M$ in either direction.

There are several ways of coding geodesics. Here we present two of them.

{\bf{The geometric code.}} As we just saw, the sides of $F$ are identified with the generators of $\text{PSL}(2, \mathbb{Z})$,
$T(z)= z+1$ and $S(z)= -1/z$. The \emph{geometric code} (also known as Morse code) of a geodesic in $F$ (that does not go to the cusp in either direction)
is a bi-infinite sequence of integers. The idea is to code the geodesic by recording the sides of the region $F$ that are cut by the geodesic.
The boundary of $F$ has three sides, the left and right vertical sides are labeled by $T$ and $T^{-1}$, respectively. The circular boundary is labeled by $S$.  Any oriented geodesic, that does not go to the cusp, returns to the circular boundary of $F$ infinitely often. The geometric code is obtained as follows. Choose an initial point on the circular boundary of $F$ and count the number of times it hits the vertical boundary of $F$ moving  in the direction of the geodesic. We assign a positive integer to each block of hits of the right vertical side and a negative number to the left vertical side. if we move the initial point in the opposite direction we obtain a sequence of nonzero integers
\[ [\gamma]=[ \dots ,n_{-1}, n_0, n_1, \dots ] \]
that we denote \emph{geometric code}.

{\bf{The arithmetic code.}} An oriented  geodesic $\gamma \in \mathcal{H}$ is called \emph{reduced} if its endpoints
$u,v$ satisfy $0<u<1$ and $w>1$. Recall that  a geodesic which does not go to the cusp in either direction is such that its end points are irrationals.  Consider the minus continued fraction associated to $u$ and $v$,
\begin{equation*}
w= \textrm{ } n_1 +\cfrac{1}{n_2 - \cfrac{1}{n_3 - \cfrac{1}{n_4 - \dots}}} \text{ , }  \frac{1}{u} = \textrm{ } n_0 -\cfrac{1}{n_{-1} - \cfrac{1}{n_{-2} - \cfrac{1}{n_{-3} - \dots}}} \end{equation*}
The following code is given to $\gamma$,
\[(\gamma)= ( \dots, n_{-2}, n_{-1}, n_{0}, n_1, n_2, \dots).\]
Now, it is possible to show that an arbitrary geodesic $\gamma \in M$ can also be represented by doubly infinite sequence. The idea is to construct a cross section  for the geodesic flow on $T^1M$. Every oriented geodesic can be represented as a bi-infinite sequence of segments $\sigma_i$ between returns to the cross section. To each segment  it corresponds a reduced geodesic $\gamma_i$. It turns out that (see \cite{gk}) all these geodesics have the same arithmetic code, except for a shift. This sequence is the \emph{arithmetic code} of $\gamma$,
\[(\gamma)= (\dots, , n_{-2}, n_{-1}, n_{0}, n_1, n_2, \dots).\]

\begin{defi}
A geodesic $\gamma$ is called \emph{positive} if  all segments comprising the geodesic $\gamma$  in F are positively (clockwise) oriented. The set of vectors in $T^1M$ tangent to positive geodesics is a  non-compact invariant set of the geodesic flow on $T^1M$. We call the restriction of  the geodesic flow to this set the \emph{positive geodesic flow}.
\end{defi}
It was shown in \cite{gk} that a geodesic $\gamma$ is positive if and only if the arithmetic and the geometric codes for $\gamma$ coincide.

%

%
Constructing an appropriate cross section and computing the first return time function of the geodesic flow to it, Gurevich and Katok \cite{gk} showed that the positive geodesic flow on the modular surface can be represented by the suspension flow  $(Y^*, \Phi^*)$. This flow is defined over the countable (two-sided) Markov shift $(\Sigma_A^*, \sigma)$, where
\[\Sigma_A^*:=\left\{ (x_n)_{n \in \Z} : A(x_n, x_{n+1})=1  \text{ for every } n \in \Z \right\}.\]
Here $A$ denotes  the transition matrix defined by $A(i,j)=1$ for every pair $(i,j) \in \{3,4,5,6,\dots \} \times\{3,4,5,6,\dots \}$ except for the pairs  corresponding to the five platonic bodies, that is
$(i,j) \in \{(3,3) , (3,4), (3,5), (4,3), (5,3) \}$, where $A(i,j)=0$.  The roof function is given by $\tau^*(x)= 2 \log w(x)$, where
\[w(x)= \textrm{ } n_1 +\cfrac{1}{n_2 - \cfrac{1}{n_3 - \cfrac{1}{n_4 - \dots}}}, \]
for $x=(\dots n_{-1}, n_0, n_1, n_2 , n_3 \dots)$.
Since every potential $\overline{\rho}:\Sigma_A^* \to \R$ is cohomologous to a potential
 $\rho:\Sigma_A^* \to \R$ that only depends on \emph{future} coordinates, we can reduce the study of the suspension flow to that of the suspension semi-flow.  Therefore, in order to study the ergodic theory of the positive geodesic flow it is enough to understand the ergodic theory of the suspension semi-flow $(Y,\Phi)$ defined over the one-sided countable Markov shift $(\Sigma_A, \sigma)$ with height function $\tau:\Sigma_A \to \R$ defined by $\tau(x)= 2\log w(x)$. Note that $(\Sigma_A, \sigma)$ satisfies the BIP property and that if $c= (3 + \sqrt{5})/6$ then for $x=(\dots n_{-1}, n_0, n_1, n_2 , n_3 \dots)$ we have that $2 \log (c n_1) \leq \tau(x) \leq 2 \log n_1$.

Since we want to study the thermodynamic formalism, we start by defining the class of potentials that we will consider. Denote by
\begin{equation*}
\mathcal{P}:= \left\{g:Y \to \R: \text{the potential } \Delta_g \text{ is locally H\"older}   \right\}.
\end{equation*}

In the next Theorem we establish conditions under which a potential has an equilibrium measure. This condition is a \emph{small oscillation} type of condition. This result first appeared in the note not intended for publication \cite{I2}.

\begin{teo}
Let $g \in \mathcal{P}$ be a bounded potential with $P_{\Phi}(g) < \infty$. If
\begin{equation} \label{va}
\sup g - \inf g < h_{top}(\Phi) - \frac{1}{2},
\end{equation}
then $g$ has an equilibrium measure $\mu_g    \in \M_{\Phi}$.
 \end{teo}

\begin{proof}
Let us first show that $P_\sigma(\Delta_g -P_{\Phi}(g)\tau)=0$.
Denote by $s= \inf F$ and by $S=\sup F$. We have that $s\tau\le\Delta_g\le S\tau$. Moreover,
\[
P_\sigma((s-t)\tau) \le P_\sigma(\Delta_F -t\tau) \le
P_\sigma((S-t)\tau).
\]
Let $t^* \in \left(1/2 +S , h_{top}(\Phi)+s \right)$ (by equation \eqref{va} this interval is non degenerate). We have that
\[
0< P_\sigma((s-t^*)\tau) \le P_\sigma(\Delta_g -t^*\tau) \le
P_\sigma((S-t^*)\tau) < \infty.
\]
Since $P_{\Phi}(g) < \infty$ and by the continuity of the function $t \to  P_\sigma(\Delta_g -t\tau)$ we have that $P_\sigma(\Delta_g -P_{\Phi}(g)\tau)=0$.

It remains to show that the potential $\Delta_g -P_{\Phi}(g)\tau $ has an equilibrium measure. Recall that there exists a Gibbs measure $\mu$ associated to this potential. In order to show that this measure is an equilibrium measure it suffices to prove that
\begin{equation} \label{der}
\int \left(  \Delta_g -P_{\Phi}(g)\tau \right) \text{d}\mu < \infty.
\end{equation}
But note that \cite[Chapter 4]{PU}
\[ \frac{\text{d}}{\text{d}t}  P_\sigma(\Delta_g -t\tau) \Big{|}_{t=P_{\Phi}(g)}=  \int \left(  \Delta_g -P_{\Phi}(g)\tau \right) \text{d}\mu. \]
Note that we have proved that there exists an interval of the form $[P_{\phi}(g)-\epsilon, P_{\phi}(g)+\epsilon]$ where the function $t \to P_\sigma(\Delta_g -t\tau)$ is finite.
The result now follows, because when finite the function $t \to P_\sigma(\Delta_g -t\tau)$ is real analytic.
\end{proof}

Let us remark that in a  wide range of (discrete time) settings, similar small oscillation conditions on the potential have been imposed in order to prove existence and uniqueness of equilibrium measures. For instance, Hofbauer \cite{ho} originally in a symbolic setting  and later Hofbauer and Keller   \cite{hk}
in the context of the angle doubling map on the circle, gave examples which shows that  this type of condition was essential in their setting in order to have quasi-compactness of the transfer operator, and hence good equilibrium measures. Later, Denker and Urba\'nski \cite{du}, for rational maps, used similar conditions to prove uniqueness of equilibrium measures. Oliveira \cite{o} proved the existence of equilibrium measures for potentials satisfying similar conditions for  non-uniformly expanding maps.  Also, Bruin and Todd \cite{bt} in the context of multimodal maps  made use of a similar condition to prove uniqueness of equilibrium measures. Recently, Inoquio-Renteria and Rivera-Letelier \cite{ir} proved this type of results for rational maps under a very weak small oscillation condition.

We end up applying the results of the previous sections to obtain phase transitions in this setting. It is a direct consequence of  Theorem \ref{teo1} that,

\begin{teo}
Consider the positive geodesic flow and the associated suspension semi-flow $(Y, \Phi)$ with roof function $\tau$. Let $g \in \mathcal{P}$ be a  potential such that
\begin{equation*}
 \lim_{n \to \infty} \frac{\sup \{ \Delta_g(x): x \in C_n\}}{-\log n}=0
\end{equation*}
We then have the following situation
\begin{enumerate}
\item\label{part1*}
For all $t \in \R$ we have that  $P_{\Phi}(tg)\geq 1/2$.
\item\label{part2*}
If for all $t\in\R$ we have that $P(t\Delta_g-(1/2)\tau)=\infty$ then there is no phase transition and the function $t\to P(tg)$ is real analytic.
\item\label{part3*}
If $g$ is strictly positive and there exists $t^*$ such that $P(t^*\Delta_g-(1/2)\tau)<\infty$ then there exists $t_0 \in \R$  such that $t_0=\sup\{t:P(t_0\Delta_g-(1/2)\tau)\leq 0\}$, and
\begin{eqnarray*}
P_{\Phi}(tg)=
\begin{cases}
\text{real analytic and strictly convex } & \text{ if } t>t_0;\\
\frac{1}{2}& \text{ if } t \leq t_0.
\end{cases}
\end{eqnarray*}
\end{enumerate}
\end{teo}

\subsection{Inducing schemes}
The results discussed in the previous sections hold under a combinatorial assumption on the base shift, namely that it satisfies the BIP condition. A well known procedure to deal with systems not having this property is to induce. Let $(\Sigma, \sigma)$ be a topologically mixing countable Markov shift.  Fix a symbol in the alphabet, say $a \in \N$.
The \emph{induced system} over the cylinder $C_a$ is the full-shift defined on the alphabet
\begin{equation*}
\left\{ C_{ai_1 \dots i_m} : i_j \neq a \text{ and }  C_{ai_1 \dots i_ma} \neq \emptyset \right\}.
\end{equation*}
The \emph{first return time map} to the cylinder $C_a$ is defined by
\[ r_a(x):= \chi_{C_a}(x) \inf \left\{ n \geq 1: \sigma^n(x) \in C_a \right\},\]
where $\chi_{C_a}(x) $ is the characteristic function of the cylinder $C_a$. For every
potential $\phi:\Sigma \to \R$ we define the induced potential by
\[\overline{\phi}:=\left( \sum_{k=0}^{r_a(x)-1} \phi \circ \sigma^k \right).\]
Note that if $\phi$ has summable variations then it is also the case for $\overline{\phi}$. In the same way, if $\phi$ is weakly H\"older then so is $\overline{\phi}$.
Since the original system is topologically mixing the pressure $P(\overline{\phi})$ does not depends on the symbol we induce, see \cite[Lemma 2]{sa2}.

\begin{lema}
Let  $(Y, \Phi)$ be a suspension semi-flow defined over  a topologically mixing countable Markov shift $(\Sigma , \sigma)$ with roof function $\tau$ of summable variations. Let $g:Y \to \R$ be such that $\Delta_g$ is of summable variations.  We have that
\begin{equation*}
P(\Delta_g - s \tau) \leq 0 \quad \text{ if and only if } \quad P(\overline{\Delta_g} - s \overline{\tau}) \leq 0.
\end{equation*}
\end{lema}

\begin{proof}
Assume that there exists $A>0$ such that
\[  P(\overline{\Delta_g} - s \overline{\tau}) \geq A >0.\]
By the approximation property, this means that for every $\epsilon >0$ there exists a compact invariant  set  $\Sigma_n \subset \Sigma$ such that
\[ P_{\Sigma_n}(\overline{\Delta_g} - s \overline{\tau}) > A - \epsilon >0.\]
In particular there exists an invariant measure in $\mu_n \in \Sigma_n$ such that
\[h(\mu_n) +\int \overline{\Delta_g} \text{d}\mu_n - s \int \overline{\tau} \text{d}\mu_n > A - 2\epsilon >0.\]
Since the measure $\mu_n$ is supported on the compact set $\Sigma_n$ the return time function is such that $r \in L^1(\mu_n)$. Therefore
\begin{eqnarray*}
\frac{h(\mu_n)}{\int r \text{d} \mu_n} +\frac{\int \overline{\Delta_g} \text{d} \mu_n} {\int r \text{d} \mu_n}- s \frac{\int \overline{\tau} \text{d} \mu_n}{\int r \text{d} \mu_n} > \frac{A - 2\epsilon}{\int r \text{d} \mu_n} >0.
\end{eqnarray*}
If we denote by $\nu_n$ the projection onto $\Sigma$ of the measure $\mu_n$ we have that
\begin{eqnarray*}
P(\Delta_g - s \tau) \geq h(\nu_n) +\int \Delta_g \text{d} \nu_n - s \int \tau \text{d} \nu_n= & \\
\frac{h(\mu_n)}{\int r \text{d} \mu_n} +\frac{\int \overline{\Delta_g} \text{d} \mu_n} {\int r \text{d} \mu_n}- s \frac{\int \overline{\tau} \text{d} \mu_n}{\int r \text{d} \mu_n} > \frac{A - 2\epsilon}{\int r \text{d} \mu_n} >0.
\end{eqnarray*}

Let us assume now that $P(\Delta_g - s \tau) > A >0$. Then by the approximation property there exists a compactly supported measure $\nu_n$ in $\Sigma$ which is ergodic, for which $\nu(C_0)>0$ and the corresponding induced measure $\mu_n$, induced on the symbol $0$,  satisfies
$\int r \text{d}\mu_n < \infty$ and
\begin{eqnarray*}
P(\Delta_g - s \tau) \geq h(\nu_n) +\int \Delta_g \text{d} \nu_n - s \int \tau \text{d} \nu_n= & \\
\frac{h(\mu_n)}{\int r \text{d} \mu_n} +\frac{\int \overline{\Delta_g} \text{d} \mu_n} {\int r \text{d} \mu_n}- s \frac{\int \overline{\tau} \text{d} \mu_n}{\int r \text{d} \mu_n} > \frac{A - \epsilon}{\int r \text{d} \mu_n} >0.
\end{eqnarray*}
Indeed, we just need to choose the measure $\nu_n$ supported on a sufficiently large  compact invariant subset of $\Sigma$.  Multiplying the last equation by $\int r d \mu_n$ we obtain
\[P(\overline{\Delta_g} - s \overline{\tau}) \geq h(\mu_n) +\int \overline{\Delta_g} \text{d} \mu_n - s \int \overline{\tau} \text{d} \mu_n > A - \epsilon >0.\]
 \end{proof}

The above Lemma implies that we can compute the value $P_{\Phi}(g)$ either on the original system $(\Sigma , \sigma)$ or in its induced version $(\overline{\Sigma}, \overline{\sigma})$.
Therefore, we can apply the results obtained in the previous sections.  We define $s_{\infty} \in \R$ as the unique value such that
\begin{equation*}
 P(-s\overline{\tau})=
\begin{cases}
\infty & \text{ if } s <s_{\infty};\\
\text{finite} & \text{ if } s > s_{\infty}.
\end{cases}
\end{equation*}

\begin{teo} \label{ind}
Let  $(Y, \Phi)$ be a suspension semi-flow defined over  a topologically mixing countable Markov shift $(\Sigma , \sigma)$ with locally H\"{o}lder roof function $\tau$. Let $g:Y \to \R$ be such that $\Delta_g$ is locally H\"{o}lder.  Assume that
\begin{equation*}
\inf \left\{ \frac{\int \overline{\Delta_g} \text{d} \nu}{\int \overline{\tau} \text{d} \nu} : \nu \in \M_{\sigma} \right\} = \lim_{n \to \infty} \frac{\sup \{ \overline{\Delta_g(x)}: x \in C_n\}}{\inf \{ \overline{\tau(x)}: x \in C_n\}}=0
\end{equation*}
We then have the following situation
\begin{enumerate}
\item\label{Part1}
For all $t \in \R$ we have that  $P_{\Phi}(tg)\geq s_{\infty}$.
\item\label{Part2}
If for all $t\in\R$ we have that $P(t\overline{\Delta_g}-s_{\infty}\overline{\tau})=\infty$ then there is no phase transition and the function $t\to P(tg)$ is real analytic.
\item\label{Part3}
If $\overline{\Delta_g}$ is a strictly positive function and there exists $t^*$ such that $P(t^*\overline{\Delta_g}-s_{\infty}\overline{\tau})<\infty$ then there exists $t_0 \in \R$ such that $t_0=\sup\{t:P(t_0\overline{\Delta_g}-s_{\infty}\overline{\tau})\leq 0\}$, and
\begin{eqnarray*}
P_{\Phi}(tg)=
\begin{cases}
\text{real analytic and strictly convex } & \text{ if } t>t_0;\\
s_{\infty}& \text{ if } t \leq t_0.
\end{cases}
\end{eqnarray*}
\end{enumerate}
\end{teo}

\end{document}